\documentclass{amsart}
\usepackage{pdfpages}
\newtheorem{theorem}{Theorem}[section]

\newtheorem{corollary}[theorem]{Corollary}
\theoremstyle{definition}
\newtheorem{definition}[theorem]{Definition}

\newtheorem{proposition}{Proposition}
\theoremstyle{remark}

\numberwithin{equation}{section}

\begin{document}

\title{Normal curves on a smooth immersed surface}

\author{Absos Ali Shaikh$^1$, Mohamd Saleem Lone$^2$ and Pinaki Ranjan Ghosh$^{3}$}
\address{$^1$Department of Mathematics, University of
Burdwan, Golapbag, Burdwan-713104, West Bengal, India}
\email{aask2003@yahoo.co.in, aashaikh@math.buruniv.ac.in}
\address{$^2$International Centre for Theoretical Sciences, Tata Institute of Fundamental Research, 560089, Bengaluru, India}
\email{saleemraja2008@gmail.com, mohamdsaleem.lone@icts.res.in}
\address{$^3$Department of Mathematics, University of
Burdwan,Golapbag, Burdwan-713104, West Bengal, India}
\email{mailtopinaki94@gmail.com}





\begin{abstract}
The aim of this paper is to investigate the sufficient condition for the invariance of a normal curve on a smooth immersed surface under isometry. We also find the the deviations of the tangential and normal components  of the curve with respect to the given isometry. 
\end{abstract}

\maketitle
\section{Introduction}
When we talk of a manifold, one of the most elementary geometric object is its position vector field, which reveals the position of an arbitrary point on that manifold with respect to some origin. In case of curves, the position vector field of a curve can be thought of as the motion of a particle with respect to a parameter $s$ and intuitively the first and second derivatives of the curve gives the velocity and the acceleration of the particle respectively. So, we shall discuss a problem which altogether depends upon the position vector field of a curve.

Let $\alpha:I\subset \mathbb{R}\rightarrow \mathbb{E}^3$ be a unit speed curve having all the necessary properties such that $\{t,n,b\}$ acts as its Serret-Frenet frame, where $t,n$, and $b$ are the tangent, the normal and the binormal vectors(unitary), respectively. Then, the Serret-Frenet equations are given by 
\begin{eqnarray}\label{as}
\left\{
\begin{array}{ll}
t^\prime =\kappa n\\
n^\prime = -\kappa t +\tau b\\
b^\prime =-\tau n,
\end{array}
\right.
\end{eqnarray}
where $\kappa$ is the curvature and $\tau$ is the torsion of $\alpha$ with $t=\alpha^\prime$, $n=\frac{t^\prime}{\kappa}$, $b=t\times n$, and $\prime$ denotes the differentiation with respect to the parameter $s$.
 At each point $\alpha(s)$ of $\alpha$, the planes spanned by $\{t,n\}$, $\{t,b\}$ and $\{n,b\}$ are called as the osculating plane, the rectifying plane and the normal plane, respectively. As it is evident from the names of planes, a curve whose position vector field lies in the osculating plane is called as an osculating curve. Similarly, a curve whose position vector filed lies in the rectifying and the normal plane are called as rectifying and normal curves, respectively.
 
It is well known that if at each point the position vector of $\alpha$ lies in the osculating plane, then the curve lies in a plane. Similarly, if the the position vector of $\alpha$ lies in the normal plane at each point, then the curve lies on a sphere. So, in view of these facts, in 2003 Chen(\cite{2}) posed a question: When does the position vector of a curve lies in the rectifying curve? Therein, Chen obtained characterization results for rectifying curves. In relation to these three types of curves, enormous study has been done. For generic study, we refer the reader to see(\cite{BYC05,3,7g,7f}).

Before going to motivation and objective of this paper, we revisit a few definitions (\cite{1}) on surfaces.
\begin{definition} A diffeomorphism $J:{\bf M}\rightarrow \overline{{\bf M}}$ is an isometry if for all $p \in {\bf M }$ and all pairs  $x_1,x_2 \in T_p({\bf M})$, we have
$$\langle x_1, x_2 \rangle_p =\langle dJ_p(x_1),dJ_p(x_2) \rangle_{J(p)}.$$
The surfaces ${\bf M}$ and $\overline{{\bf M}}$ are then said to be isometric.
\end{definition}
\begin{definition}
A map $J:V \subset {\bf M} \rightarrow \overline{{\bf M}}$ of a neighborhood $V$ of $p \in {\bf M}$ is a local isometry at $p$ if there exists a neighborhood $U$ of $J(p) \in \overline{\bf M}$ such that $J: V \rightarrow U$ is an isometry. If there exists local isometry at every point of ${\bf M}$, then the surfaces ${\bf M}$ and $\overline{{\bf M}}$ are said to be locally isometric. Clearly, if $J$ is a diffeomorphism and a local isometry for every $p \in {\bf M}$, then $J$ is global isometry.
\end{definition}
It is straightforward to see that the first fundamental form coefficients  are preserved under isometry. So if $E,F,G$ and $\overline{E},\overline{F},\overline{G}$ are the first fundamental form coefficients of ${\bf M}$ and $\overline{{\bf M}}$, respectively and $J:{\bf M}\rightarrow \overline{{\bf M}}$ is a local isometry then 
\begin{equation*}E=\overline{E},\quad F=\overline{F},\quad G=\overline{G}.\end{equation*}

In 2018 Shaikh and Ghosh (\cite{4}) diverted the study of rectifying curves to a new direction by questioning about the invariant properties of a rectifying curve on a smooth surface under isometry. In addition to a sufficient condition for a rectifying curve to remain invariant under isometry, they showed that the component of the rectifying curve along the surface normal is invariant under isometry. Again in (\cite{AP2}) Shaikh and Ghosh studied osculating curves and obtained their characterization along with invariancy under surface isometry. Motivated by (\cite{4}, \cite{AP2} and \cite{AP3}), we shall investigate the similar questions in case of {\it normal curves}, i.e.,

{\it Question:} What happens to a normal curve on a smooth surface under isometry?

In the section $2$, we give some of the basic notions about normal curves  and find the Frenet frame vectors of the normal curves with respect to the smooth immersed surface. Section 3 is concerned with the main results and provided the answer of above question.
\section{Preliminaries}
Let $\alpha$ be a normal curve parameterized by arc with a Serret-Frenet frame given in (\ref{as}). The other way of interpreting a normal curve is: a curve is said to be a normal curve if its position vector lies in the orthogonal complement of tangent vector i.e., $\alpha \cdot t =0,$ or
\begin{equation}\label{1}
\alpha(s)=\lambda(s)n(s)+ \mu(s)b(s),
\end{equation}
where $\lambda,\mu$ are two smooth functions.

Suppose ${\bf M}$ is a regular surface(page no 52, \cite{1}) with 
$\varphi(u,v)$ being its coordinate chart. Then, the curve 
$\alpha(s)=\alpha(u(s),v(s))$ defines a curve 
$ \alpha(s)= {\bf M}(u(s),v(s))$ on the surface ${\bf M}$. We can easily find the derivatives of the curve $\alpha(s)$ as a curve on the surface ${\bf M}$ using the chain rule:
\begin{eqnarray}
\nonumber\alpha^\prime(s)&=&\varphi_uu^\prime+\varphi_vv^\prime\\
\nonumber\text{or, }&&\\
\label{2} t(s)&=&\alpha^\prime(s)=\varphi_uu^\prime+\varphi_vv^\prime\\
\nonumber
 {t'}(s)&=& u^{\prime\prime}\varphi_u+v^{\prime\prime}\varphi_v+{u^\prime}^2\varphi_{uu}+2u^\prime v^\prime \varphi_{uv}+{v^\prime}^2\varphi_{vv}.
\end{eqnarray}
Now let ${\bf N}$ be the unit surface normal then we have
\begin{eqnarray}\label{3}
\nonumber
n(s)&=&\frac{1}{k(s)}(u''\varphi_u+v''\varphi_v+u'^2\varphi_{uu}+2u'v'\varphi_{uv}+v'^2\varphi_{vv}).\\
\nonumber
b(s)&=& t(s)\times n(s)= t(s)\times \frac{t'(s)}{k(s)}\\
\nonumber
&=&\frac{1}{k(s)}\Big[(\varphi_uu'+\varphi_vv')\times(u''\varphi_u+v''\varphi_v+u'^2\varphi_{uu}+2u'v'\varphi_{uv}+v'^2\varphi_{vv})\Big]\\
\nonumber&=&\frac{1}{k(s)}\Big[\{u'v''-u''v'\}{\bf N}+u'^3\varphi_u\times \varphi_{uu}+2u'^2v'\varphi_u\times \varphi_{uv}+u'v'^2\varphi_u\times \varphi_{vv}\\
&&+u'^2v'\varphi_v\times \varphi_{uu}+2u'v'^2\varphi_v\times \varphi_{uv}+v'^3\varphi_v\times \varphi_{vv}\Big].
\end{eqnarray} 
\begin{definition}\cite{AP01}
Let $\alpha$ be a unit speed curve on ${\bf M}$, then the unit tangent vector $t=\alpha^\prime$ is orthogonal to the unit surface normal ${\bf N}$, so $\alpha^\prime$, ${\bf N}$ and ${\bf N} \times \alpha^\prime$ are mutually orthogonal vectors. Moreover, since $\alpha^{\prime \prime}$ is orthogonal to $\alpha^\prime$, we can write $\alpha^{\prime\prime}$  as a linear combination of ${\bf N}$ and ${\bf N} \times \alpha^\prime$, i.e.,
\begin{equation*}
\alpha^{\prime\prime}=\kappa_{n}{\bf N}+\kappa_g{\bf N}\times \alpha^\prime,
\end{equation*}
where $\kappa_n$ and $\kappa_g$ are called as the normal curvature and the geodesic curvature of $\alpha$, respectively and are given by
\begin{eqnarray*}
\left\{
\begin{array}{ll}
\kappa_g =\alpha^{\prime \prime}\cdot {\bf N} \times \alpha^\prime \\
\kappa_n =\alpha^{\prime \prime}\cdot {\bf N}.
\end{array}
\right.
\end{eqnarray*}
Now since $\alpha^{\prime \prime}=\kappa(s)n(s)$, therefore we can write
\begin{equation*}
\kappa_n=\kappa(s)n(s)\cdot {\bf N}=(u''\varphi_u+v''\varphi_v+u'^2\varphi_{uu}+2u'v'\varphi_{uv}+v'^2\varphi_{vv})\cdot{\bf N}
\end{equation*}
or 
\begin{equation}\label{se1}
\kappa_n = {u^\prime}^2 L + 2u^\prime v^\prime M +{v^\prime}^2N,
\end{equation}
where $L,M,N$ are the second fundamental form coefficients of the surface. The curve $\alpha$ on $\bf{M}$ is said to be asymptotic if and only if $\kappa_n=0.$
\end{definition}
\section{Normal curves}
The equation of a normal curve is given by
\begin{equation}\label{1}
\alpha(s)=\lambda(s)n(s)+ \mu(s)b(s).
\end{equation}
Suppose this curve lies on a parametric surface $\varphi(u,v)$.
Then (\ref{1}) is in the form: 
\begin{eqnarray}\label{1.2}
\nonumber\alpha(s)&=&\frac{\lambda(s)}{\kappa(s)}\left[(u^{\prime\prime}\varphi_u+v^{\prime\prime} \varphi_v)+({u^\prime}^2 \varphi_{uu}+2u^\prime v^\prime \varphi_{uv}+{v^\prime}^2 \varphi_{vv})\right]\\
&&+\frac{\mu(s)}{k(s)}\Big[\{u'v''-u''v'\}\vec{N}+u'^3\varphi_u\times \varphi_{uu}+2u'^2v'\varphi_u\times \varphi_{uv}\\
\nonumber
&&+u'v'^2\varphi_u\times \varphi_{vv}+u'^2v'\varphi_v\times \varphi_{uu}+2u'v'^2\varphi_v\times \varphi_{uv}+v'^3\varphi_v\times \varphi_{vv}\Big].
\end{eqnarray}
\begin{theorem}
Let ${\bf M}$ and ${\overline{{\bf M}}}$ be two smooth surfaces and $J: {\bf M}\rightarrow {\overline{\bf M}}$ be an isometry. Also, let $\alpha(s)$ be a normal curve on ${\bf M}$. Then $\overline{\alpha}(s)=J \circ \alpha(s)$ is a normal curve on 
$\overline{{\bf M}}$ if
\begin{eqnarray}
\nonumber \overline{\alpha}(s)-J_{\ast}(\alpha(s))&=&\frac{\lambda(s)}{\kappa(s)}\Big[{u^\prime}^2\left(\frac{\partial J_\ast}{\partial u}\varphi_u\right)+2u^\prime v^\prime \left(\frac{\partial J_\ast}{\partial u}\varphi_v \right)+{v^\prime}^2\left(\frac{\partial J_\ast}{\partial v}\varphi_v\right)\Big]\\
&&+\frac{\mu(s)}{k(s)}\Big[u'^3\Big(J_*\varphi_u\times \frac{\partial J_*}{\partial u}\varphi_u\Big)+2u'^2v'\Big(J_*\varphi_u\times \frac{\partial J_*}{\partial u}\varphi_v\Big)\\
\nonumber
&&+u'v'^2\Big(J_*\varphi_u\times \frac{\partial J_*}{\partial v}\varphi_v\Big)+u'^2v'\Big(J_*\varphi_v\times \frac{\partial J_*}{\partial u}\varphi_u\Big)\\
\nonumber
&&+2u'v'^2\Big(J_*\varphi_v\times \frac{\partial J_*}{\partial u}\varphi_v\Big)+v'^3\Big(J_*\varphi_v\times \frac{\partial J_*}{\partial v}\varphi_v\Big)\Big].
\end{eqnarray}
\end{theorem}
\begin{proof}
Suppose $\varphi$ and $\overline{\varphi}$ are the chart maps of ${\bf M}$ and ${\overline{\bf M}}$, respectively. Then, we have
$$\overline{\varphi}=J\circ \varphi.$$
$J:{\bf M}\rightarrow {\overline{\bf M}}$ is an isometry whose differential map $dJ=J_\ast$ is a $3\times 3$ orthogonal matrix taking linearly independent vectors of $T_p({\bf M})$ to linearly independent vectors of $T_{J(p)}{\overline{ \bf M}}$, i.e., 
$$J_\ast : T_p({\bf M})\rightarrow T_{J(p)}{\overline{ \bf M}}.$$
Since $\{\varphi_u , \varphi_v\}$ is a basis of tangent plane $T_p({\bf M})$ at a point $p$ on ${\bf M}$, we have
\begin{eqnarray}
\label{1.4} &&\bar{\varphi}_u(u,v)=J_*\varphi_u=J_*(\varphi(u,v))\varphi_u,\\
\label{1.5}&&\bar{\varphi}_v(u,v)=J_*\varphi_v=J_*(\varphi(u,v))\varphi_v.
\end{eqnarray}
Differentiating (\ref{1.4}) and (\ref{1.5}) with respect to $u,v$, we get
\begin{eqnarray}\label{1.6}
\nonumber \bar{\varphi}_{uu}&=& \frac{\partial J_*}{\partial u}\varphi_u+J_*\varphi_{uu},\\
\bar{\varphi}_{vv}&=& \frac{\partial J_*}{\partial v}\varphi_v+J_*\varphi_{vv},\\
\nonumber \bar{\varphi}_{uv}&=& \frac{\partial J_*}{\partial u}\varphi_v+J_*\varphi_{uv}= \frac{\partial J_*}{\partial v}\varphi_u +J_*\varphi_{uv}.\\\nonumber
 \end{eqnarray}
We can write
\begin{equation}\label{1.7}
J_*\varphi_u\times \frac{\partial J_*}{\partial u}\varphi_u=J_*\varphi_u\times\Big(\frac{\partial J_*}{\partial u}\varphi_u+F*\varphi_{uu}\Big)-J_*(\varphi_u\times \varphi_{uu})=\bar{\varphi}_u\times \bar{\varphi}_{uu}-J_*(\varphi_u\times \varphi_{uu}).
\end{equation}
Similarly 
\begin{eqnarray}\label{1.8}
\nonumber J_*\varphi_u\times \frac{\partial J_*}{\partial u}\varphi_v&=\bar{\varphi}_u\times \bar{\varphi}_{uv}-J_*(\varphi_u\times \varphi_{uv}),\\
\nonumber J_*\varphi_u\times \frac{\partial J_*}{\partial v}\varphi_v&=\bar{\varphi}_u\times \bar{\varphi}_{vv}-J_*(\varphi_u\times \varphi_{vv}),\\
J_*\varphi_v\times \frac{\partial J_*}{\partial u}\varphi_u&=\bar{\varphi}_v\times \bar{\varphi}_{uu}-J_*(\varphi_v\times \varphi_{uu}),\\
\nonumber J_*\varphi_v\times \frac{\partial J_*}{\partial u}\varphi_v&=\bar{\varphi}_v\times \bar{\varphi}_{uv}-J_*(\varphi_v\times \varphi_{uv}),\\
\nonumber J_*\varphi_v\times \frac{\partial J_*}{\partial v}\varphi_v&=\bar{\varphi}_v\times \bar{\varphi}_{vv}-J_*(\varphi_v\times \varphi_{vv}).
 \end{eqnarray}
Therefore, with respect to $(3.2)$, (\ref{1.7}) and (\ref{1.8}), we get 
\begin{eqnarray}
\nonumber \overline{\alpha}(s)&=&\frac{\lambda(s)}{\kappa(s)}\Big[u^{\prime\prime}J_{\ast}\varphi_u +v^{\prime \prime}J_{\ast}\varphi_v +{u^\prime}^2\left(\frac{\partial J_\ast}{\partial_u}\varphi_u + J_\ast \varphi_{uu}\right)+2u^\prime v^\prime \left(\frac{\partial J_\ast}{\partial u}\varphi_v + J_\ast \varphi_{uv}\right)\\
\nonumber &&+{v^\prime}^2\left(\frac{\partial J_\ast}{\partial v}\varphi_v + J_\ast \varphi_{vv}\right)\Big]
+\frac{\mu(s)}{k(s)}\Big[\{u'v''-u''v'\}J_*\vec{N}+u'^3J_*(\varphi_u\times \varphi_{uu})\\
\nonumber
&&+2u'^2v'J_*(\varphi_u\times \varphi_{uv})+u'v'^2J_*(\varphi_u\times \varphi_{vv})+u'^2v'J_*(\varphi_v\times \varphi_{uu})+2u'v'^2J_*(\varphi_v\times \varphi_{uv})\\
\nonumber
&&+v'^3J_*(\varphi_v\times \varphi_{vv})+u'^3\Big(J_*\varphi_u\times \frac{\partial J_*}{\partial u}\varphi_u\Big)+2u'^2v'\Big(J_*\varphi_u\times \frac{\partial J_*}{\partial u}\varphi_v\Big)\\
\nonumber
&&+u'v'^2\Big(J_*\varphi_u\times \frac{\partial J_*}{\partial v}\varphi_v\Big)+u'^2v'\Big(J_*\varphi_v\times \frac{\partial J_*}{\partial u}\varphi_u\Big)+2u'v'^2\Big(J_*\varphi_v\times \frac{\partial J_*}{\partial u}\varphi_v\Big)\\
\nonumber
&&+v'^3\Big(J_*\varphi_v\times \frac{\partial J_*}{\partial v}\varphi_v\Big)\Big],
\end{eqnarray}
or
\begin{eqnarray}
\nonumber \overline{\alpha}(s)&=&\frac{\lambda(s)}{\kappa(s)}\left[(u^{\prime\prime}\overline{\varphi}_u+v^{\prime\prime} \overline{\varphi}_v)+({u^\prime}^2 \overline{\varphi}_{uu}+2u^\prime v^\prime \overline{\varphi}_{uv}{v^\prime}^2 \overline{\varphi}_{vv})\right]\\
&&+\frac{\mu(s)}{k(s)}\Big[\{u'v''-u''v'\}\vec{\overline{N}}+u'^3\varphi_u\times \overline{\varphi}_{uu}+2u'^2v'\overline{\varphi}_u\times \overline{\varphi}_{uv}\\
\nonumber
&&+u'v'^2\overline{\varphi}_u\times \overline{\varphi}_{vv}+u'^2v'\overline{\varphi}_v\times \overline{\varphi}_{uu}+2u'v'^2\overline{\varphi}_v\times \overline{\varphi}_{uv}+v'^3\overline{\varphi}_v\times \overline{\varphi}_{vv}\Big].
\end{eqnarray} 
Therefore
\begin{equation}
\overline{\alpha}(s)=\frac{\overline{\lambda}(s)}{\overline{\kappa}(s)}\overline{n}(s)+\frac{\overline{\mu}(s)}{\overline{\kappa}(s)}\overline{b}(s)
\end{equation} for some functions $\overline{\lambda}(s)$ and $\overline{\mu}(s)$. Therefore, $\overline{\alpha}(s)$ is a normal curve in ${\overline{\bf M}}$.
\end{proof}
\begin{theorem}
Let $J:{\bf M}\rightarrow {\overline{\bf M}}$ be an isometry and $\alpha(s)$ be a normal curve on ${\bf M}$. Then for the tangential components, we have
\begin{eqnarray}\label{3.2.A}
\bar{\alpha}\cdot {\bf\bar{T}}-\alpha\cdot {\bf T}=\frac{\mu}{\kappa}(\bar{\kappa}_n-\kappa_n)(av^\prime+bu^\prime),
\end{eqnarray}
where ${\bf T}=a\varphi_u+b\varphi_v$ is any tangent vector to ${\bf M}$ for some $a,b\in \mathbb{R}$.
\end{theorem}
\begin{proof}
From (\ref{1.2}), we see that
\begin{eqnarray}
\nonumber  \alpha \cdot  \varphi_u &=&\frac{\lambda}{\kappa}\Big[u^{\prime \prime} E  +v^{\prime \prime} F+{u^\prime}^2 \varphi_{uu} \cdot  \varphi_u + 2u^\prime v^\prime  \varphi_{uv} \cdot  \varphi_u  + {v^\prime}^2 \varphi_{vv} \cdot  \varphi_u  \Big]\\
\label{a1.13}&&+\frac{\mu}{\kappa}\Big[{u^\prime}^2 v^\prime L  +2{v^\prime}^2 u^\prime  M+{v^\prime}^3  N \Big].
\end{eqnarray}
Now for the isometric images of $\alpha$ and $\varphi_u$,  we have
\begin{eqnarray}
\nonumber  \overline{\alpha} \cdot  \overline{\varphi}_u &=&\frac{\lambda}{\kappa}\Big[u^{\prime \prime} \overline{E}  +v^{\prime \prime} \overline{F}+{u^\prime}^2 \overline{\varphi}_{uu} \cdot  \overline{\varphi}_u + 2u^\prime v^\prime  \overline{\varphi}_{uv} \cdot  \overline{\varphi}_u  + {v^\prime}^2 \overline{\varphi}_{vv} \cdot  \overline{\varphi}_u  \Big]\\
\label{a1.14}&&+\frac{\mu}{\kappa}\Big[{u^\prime}^2 v^\prime\overline{ L}  +2{v^\prime}^2 u^\prime  \overline{M}+{v^\prime}^3  \overline{N} \Big].
\end{eqnarray}
Since we know that $\overline{E}=E,\overline{F}=F,\overline{G}=G$. In particular
$$E=\overline{E}=J_\ast \varphi_u \cdot J_\ast \varphi_u=\varphi_u \cdot \varphi_u$$
Differentiating the above equation with respect to $u$, we get
\begin{equation*}
\left(\frac{\partial J_\ast}{\partial u}\varphi_u + J_\ast \varphi_{uu}\right)\cdot (J_\ast \varphi_u)=\varphi_{uu}\cdot \varphi_u
\end{equation*}
or
\begin{equation*}
\overline{\varphi}_{uu} \cdot \overline{\varphi}_u= \varphi_{uu} \cdot \varphi_u.
\end{equation*}
Similarly, it is easy to show that
\begin{equation*}
\overline{\varphi}_{uv}\cdot\overline{\varphi}_u=\varphi_{uv}\cdot \varphi_u, \quad
\overline{\varphi}_{vv}\cdot\overline{\varphi}_u=\varphi_{vv}\cdot \varphi_u.
\end{equation*}
Thus from (\ref{a1.14}), we get
\begin{eqnarray*}
\overline{\alpha} \cdot \overline{\varphi}_u &=&\frac{\lambda}{\kappa}\Big[u^{\prime\prime}E+v^{\prime\prime}F+{u^\prime}^2 \varphi_{uu}\cdot\varphi_u +2u^\prime v^\prime \varphi_{uv}\cdot \varphi_u +{v^\prime}^2\varphi_{vv}\cdot \varphi_u\Big]\\
&&+\frac{\mu}{\kappa}\Big[{u^\prime}^2v^\prime\overline{L}+2{v^\prime}^2u^\prime \overline{M}+{v^\prime}^3\overline{N}\Big],
\end{eqnarray*}
or
\begin{equation}\label{a1.15}
\overline{\alpha} \cdot \overline{\varphi}_u =\frac{\lambda}{\kappa}\Big[u^{\prime\prime}E+v^{\prime\prime}F+{u^\prime}^2 \varphi_{uu}\cdot\varphi_u +2u^\prime v^\prime \varphi_{uv}\cdot \varphi_u +{v^\prime}^2\varphi_{vv}\cdot \varphi_u\Big]
+v^\prime\frac{\mu}{\kappa}\overline{\kappa}_n.
\end{equation}
Taking the difference of (\ref{a1.15}) and (\ref{a1.13}), we get 
\begin{equation}\label{3.2.a}
\overline{\alpha}\cdot \overline{\varphi}_u -\alpha\cdot \varphi_u=
v^\prime\frac{\mu}{\kappa}(\overline{\kappa}_n-\kappa_n).
\end{equation}
Similarly the following relation hold
 \begin{equation}\label{3.2.b}
 \overline{\alpha}\cdot \overline{\varphi}_v -\alpha\cdot \varphi_v=
 u^\prime\frac{\mu}{\kappa}(\overline{\kappa}_n-\kappa_n).
 \end{equation}
Now with the help of $(\ref{3.2.a})$ and $(\ref{3.2.b})$ we get
\begin{eqnarray*}
\bar{\alpha}\cdot{\bf \bar{T}}-\alpha\cdot {\bf T} &=&\bar{\alpha}\cdot(a\bar{\varphi}_u+b\bar{\varphi}_v)-\alpha\cdot(a\varphi
_u+b\varphi
_v)\\
&=& a(\overline{\alpha}\cdot \overline{\varphi}_u -\alpha\cdot \varphi_u)+b(\overline{\alpha}\cdot \overline{\varphi}_v -\alpha\cdot \varphi_v)\\
&=& av^\prime\frac{\mu}{\kappa}(\overline{\kappa}_n-\kappa_n)+bu^\prime\frac{\mu}{\kappa}(\overline{\kappa}_n-\kappa_n)\\
&=&
\frac{\mu}{\kappa}(\bar{\kappa}_n-\kappa_n)(av^\prime+bu^\prime).
\end{eqnarray*}
This proves our claim.
\end{proof}
\begin{corollary}\label{crly3.3}
Let $J:{\bf M}\rightarrow {\overline{\bf M}}$ be an isometry and $\alpha(s)$ be a normal curve on ${\bf M}$. Then the component of the normal curve $\alpha(s)$ along any tangent vector ${\bf T}$ to the surface ${\bf M}$ is invariant if and only if any one of the following holds:
\begin{itemize}
\item[(i)]the position vector of $\alpha(s)$ is in the normal direction of $\alpha$.
\item[(ii)] The normal curvature is invariant.
\end{itemize}
\end{corollary}
\begin{proof}
From $(\ref{3.2.A})$,  $\bar{\alpha}\cdot{\bf \bar{T}}=\alpha\cdot {\bf T}$ if and only if 
\begin{eqnarray*}
\frac{\mu}{\kappa}(\bar{\kappa}_n-\kappa_n)(av^\prime+bu^\prime)=0\\
\text{i.e., if and only if } \mu=0 \text{ or }\bar{\kappa}_n-\kappa_n=0 .
\end{eqnarray*}
If $\mu=0$ then from $(\ref{1})$, we see that $\alpha(s)=\lambda(s)n(s)$, i.e., the position vector of the normal curve $\alpha(s)$ is in the normal direction of itself.
\end{proof}

\begin{corollary}
Let $J:{\bf M}\rightarrow {\overline{\bf M}}$ be an isometry and $\alpha(s)$ be a normal curve on ${\bf M}$. The component of the normal curve $\alpha(s)$ along any tangent vector ${\bf T}$ to the surface ${\bf M}$ is invariant and the position vector of $\alpha(s)$ is not in the normal direction of $\alpha$, then $\alpha(s)$ is asymptotic if and only if $\bar{\alpha}(s)$ is asymptotic.
\end{corollary}
\begin{proof}
From Corollary $3.3.$,  $\bar{\alpha}\cdot{\bf \bar{T}}=\alpha\cdot {\bf T}$ and the position vector of $\alpha(s)$ is not in the normal direction of $\alpha$ if and only if
$\kappa_n=\bar{\kappa}_n$.\\
Therefore $\alpha(s)$ is asymptotic if and only if $\kappa_n=0$ if and only if $\bar{\kappa}_n=0$ if and only if $\bar{\alpha}(s)$ is asymptotic.
\end{proof}

\begin{theorem}
Let $J:{\bf M}\rightarrow {\overline{\bf M}}$ be an isometry and $\alpha(s)$ be a normal curve on ${\bf M}$. Then for the component of $\alpha(s)$ along the surface normal $N$, we have
\begin{equation}\label{3.4.B}
\overline{\alpha}\cdot \overline{\bf N}-\alpha \cdot {\bf N}=\frac{\lambda}{\kappa}(\overline{\kappa}_n-\kappa_n).
\end{equation}
\end{theorem}
\begin{proof}
From (\ref{1.2}), we have
\begin{eqnarray*}
\alpha \cdot {\bf N} &=&\frac{\lambda}{\kappa}\Big[{u^\prime}^2 \varphi_{uu}\cdot (\varphi_u \times \varphi_v)+{v^\prime}^2 \varphi_{vv}\cdot (\varphi_u \times \varphi_v)+2u^\prime v^\prime \varphi_{uv}\cdot (\varphi_u \times \varphi_v)\Big]\\
&&+\frac{\mu}{\kappa}\Big[(u^\prime v^{\prime\prime}-v^\prime u^{\prime\prime})(EG-F^2)+{u^\prime}^3(\varphi_u \times \varphi_{uu})\cdot (\varphi_u \times \varphi_v)\\
&&+2{u^\prime}^2v^\prime (\varphi_u \times \varphi_{uv})\cdot (\varphi_u \times \varphi_v)
+{v^\prime}^2u^\prime (\varphi_u \times \varphi_{vv})\cdot (\varphi_u \times \varphi_v)\\
&&+{u^\prime}^2v^\prime (\varphi_v \times \varphi_{uu})\cdot (\varphi_u \times \varphi_v)
+2{u^\prime}{v^\prime}^2 (\varphi_v \times \varphi_{uv})\cdot (\varphi_u \times \varphi_v)\\
&&+{v^\prime}^3 (\varphi_v \times \varphi_{vv})\cdot (\varphi_u \times \varphi_v)\Big]
\end{eqnarray*}
or
\begin{eqnarray*}
\alpha \cdot {\bf N} &=&\frac{\lambda}{\kappa}\Big[{u^\prime}^2 L+{v^\prime}^2 N+2u^\prime v^\prime M\Big]
+\frac{\mu}{\kappa}\Big[(u^\prime v^{\prime\prime}-v^\prime u^{\prime\prime})(EG-F^2)
+{u^\prime}^3\{E(\varphi_{uu}\cdot \varphi_v)\\
&&-F(\varphi_{uu}\cdot\varphi_u)\}+2{u^\prime}^2 v^\prime\{E(\varphi_{uv}\cdot \varphi_v)-F(\varphi_{uv}\cdot\varphi_u)\}+{u^\prime}{v^\prime}^2\{E(\varphi_{vv}\cdot \varphi_v)\\
&&-F(\varphi_{vv}\cdot\varphi_u)\}+{u^\prime}^2v^\prime\{F(\varphi_{uu}\cdot \varphi_v)-G(\varphi_{uu}\cdot\varphi_u)\}+2{u^\prime}{v^\prime}^2\{F(\varphi_{uv}\cdot \varphi_v)\\
&&-G(\varphi_{uv}\cdot\varphi_u)\}+{v^\prime}^3\{F(\varphi_{vv}\cdot \varphi_v)-G(\varphi_{vv}\cdot\varphi_u)\}\Big].
\end{eqnarray*}
Since we know that with respect to isometry: $E=\overline{E},F=\overline{F},G=\overline{G}$. Then, it easy to verify:
\begin{eqnarray}\label{1.15}
\left\{
\begin{array}{ll}
E_u=\overline{E}_u,F_u=\overline{F}_u,G_u=\overline{G}_u\\
E_v=\overline{E}_v,F_v=\overline{F}_v,G_v=\overline{G}_v.
\end{array}
\right.
\end{eqnarray}
Then we have $E_u= (\varphi_u \cdot \varphi_u)_u=\frac{1}{2}\varphi_{uu}\cdot\varphi_u$ or 
\begin{equation}\label{1.16}
\varphi_{uu}\cdot\varphi_u=\frac{E_u}{2}.
\end{equation}
On the similar lines, we can find
\begin{eqnarray}\label{1.17}
\left\{
\begin{array}{ll}
\varphi_{uu} \cdot \varphi_v =F_u-\frac{E_v}{2}, \varphi_{vv}\cdot \varphi_v =\frac{G_v}{2},\varphi_{vv}\cdot \varphi_u= F_v= \frac{G_u}{2}\\
\varphi_{uv}\cdot \varphi_v =\frac{G_u}{2}, \varphi_{uv}\cdot \varphi_u =\frac{E_v}{2}
\end{array}
\right.
\end{eqnarray}
 
Therefore in view of (\ref{1.16}) and (\ref{1.17}), we get
\begin{eqnarray*}
\alpha \cdot {\bf N} &=&\frac{\lambda}{\kappa}\Big[{u^\prime}^2 L+{v^\prime}^2 N+2u^\prime v^\prime M\Big]\\
&&+\frac{\mu}{\kappa}\Big[(u^\prime v^{\prime\prime}-v^\prime u^{\prime\prime})(EG-F^2)
+{u^\prime}^3\left \{E\left(F_u-\frac{E_v}{2}\right)-\frac{FE_u}{2}\right \}\\
&&+2{u^\prime}^2 v^\prime \left \{\frac{EG_u}{2}-\frac{FE_v}{2}\right \}
+{u^\prime}{v^\prime}^2 \left \{\frac{EG_v}{2}-F\left(F_v-\frac{G_u}{2}\right)\right\}\\
&&+{u^\prime}^2v^\prime \left \{\frac{FG_u}{2}-\frac{GE_u}{2}\right \}+2{u^\prime}{v^\prime}^2 \left \{FG_u - \frac{GE_v}{2} \right\}\\
&&+{v^\prime}^3 \left \{\frac{FG_v}{2}-G\left(F_v-\frac{G_u}{2}\right)\right \}\Big].
\end{eqnarray*}
Now applying $J$ and with the help of (\ref{1.15}), we get
\begin{eqnarray*}
\overline{\alpha} \cdot \overline{{\bf N}} &=&\frac{\lambda}{\kappa}\Big[{u^\prime}^2 \overline{L}+{v^\prime}^2 \overline{N}+2u^\prime v^\prime \overline{M}\Big]\\
&&+\frac{\mu}{\kappa}\Big[(u^\prime v^{\prime\prime}-v^\prime u^{\prime\prime})(EG-F^2)
+{u^\prime}^3\left \{E\left(F_u-\frac{E_v}{2}\right)-\frac{FE_u}{2}\right \}\\
&&+2{u^\prime}^2 v^\prime \left \{\frac{EG_u}{2}-\frac{FE_v}{2}\right \}
+{u^\prime}{v^\prime}^2 \left \{\frac{EG_v}{2}-F\left(F_v-\frac{G_u}{2}\right)\right\}\\
&&+{u^\prime}^2v^\prime \left \{\frac{FG_u}{2}-\frac{GE_u}{2}\right \}+2{u^\prime}{v^\prime}^2 \left \{FG_u - \frac{GE_v}{2} \right\}\\
&&+{v^\prime}^3 \left \{\frac{FG_v}{2}-G\left(F_v-\frac{G_u}{2}\right)\right \}\Big].
\end{eqnarray*}
On taking the difference of $\alpha \cdot {\bf N}$ and its isometric image and with the help of $(3.2)$, we get
\begin{equation*}
\overline{\alpha}\cdot \overline{\bf N}-\alpha \cdot {\bf N}=\frac{\lambda}{\kappa}(\overline{\kappa}_n-\kappa_n).
\end{equation*}
This proves our claim.
\end{proof}

\begin{corollary}
Let $J:{\bf M}\rightarrow {\overline{\bf M}}$ be an isometry and $\alpha(s)$ be a normal curve on ${\bf M}$. Then the component of the normal curve $\alpha(s)$ along the surface normal is invariant if and only if any one of the following holds:
\begin{itemize}
\item[(i)]The position vector of $\alpha(s)$ is in the binormal direction of $\alpha$.
\item[(ii)] The normal curvature is invariant.
\end{itemize}
\end{corollary}
\begin{proof}
From $(\ref{3.4.B})$,  $\bar{\alpha}\cdot{\bf \bar{N}}=\alpha\cdot {\bf N}$ if and only if
\begin{eqnarray*}
\frac{\lambda}{\kappa}(\overline{\kappa}_n-\kappa_n)=0\\
\text{i.e., if and only if } \lambda=0 \text{ or } \overline{\kappa}_n-\kappa_n=0.
\end{eqnarray*}
If $\lambda=0$ then from $(\ref{1})$, we see that $\alpha(s)=\mu(s)b(s)$, i.e., the position vector of the normal curve $\alpha(s)$ is in the binormal direction of itself.
\end{proof}
\begin{corollary}
Let $J:{\bf M}\rightarrow {\overline{\bf M}}$ be an isometry and $\alpha(s)$ be a normal curve on ${\bf M}$. The component of the normal curve $\alpha(s)$ along surface normal ${\bf N}$ to the surface ${\bf M}$ is invariant and the position vector of $\alpha(s)$ is not in the normal direction of $\alpha$, then $\alpha(s)$ is asymptotic if and only if $\bar{\alpha}(s)$ is asymptotic.
\end{corollary}
\begin{proof}
From Corollary $3.6.$,  $\bar{\alpha}\cdot{\bf \bar{N}}=\alpha\cdot {\bf N}$ and the position vector of $\alpha(s)$ is not in the normal direction of $\alpha$ if and only if $\kappa_n=\bar{\kappa}_n$.\\
Therefore $\alpha(s)$ is asymptotic if and only if $\kappa_n=0$ if and only if $\bar{\kappa}_n=0$ if and only if $\bar{\alpha}(s)$ is asymptotic.
\end{proof}

Since $\bf T$ and $\bf N$ are perpendicular vector at $\alpha(s)$, hence $\{\bf T,\bf N, \bf T\times\bf N\}$ form an orthogonal system at every point of the normal curve $\alpha(s)$.
\begin{theorem}
Let $J:{\bf M}\rightarrow {\overline{\bf M}}$ be an isometry and $\alpha(s)$ be a normal curve on ${\bf M}$. Then for the component of $\alpha(s)$ along $\bf T\times\bf N$, we have
\begin{equation}\label{3.6.C}
\bar{\alpha}\cdot({\bf \bar T}\times{\bf \bar N})-\alpha\cdot({\bf T}\times{\bf N}) =\frac{\mu}{\kappa}(\bar{\kappa}_n-\kappa_n) \{a(F v^\prime-Eu^\prime)+b(Gv^\prime-Fu^\prime)\}
\end{equation}
\end{theorem}
\begin{proof}
From $(\ref{2})$, we have
\begin{eqnarray*}
\alpha\cdot(\bf T\times\bf N)&=& \alpha\cdot\{(a\varphi_u+b\varphi_v)\times\bf N\},\\
&=& \alpha\cdot\{a(F\varphi_u-E\varphi_v)+b(G\varphi_u-F\varphi_v)\},\\
&=& (aF+bG)\alpha\cdot\varphi_u-(aE+bF)\alpha\cdot\varphi_v.
\end{eqnarray*}
Therefore using $(\ref{3.2.a})$ and $(\ref{3.2.b})$ we get
\begin{eqnarray*}
\bar{\alpha}\cdot(\bar{\bf T}\times\bar{\bf N})-\alpha\cdot(\bf T\times\bf N)&=& (aF+bG)(\overline{\alpha}\cdot \overline{\varphi}_u -\alpha\cdot \varphi_u)-(aE+bF)\\
&&(\overline{\alpha}\cdot \overline{\varphi}_v -\alpha\cdot \varphi_v),\\
&=& \frac{\mu}{\kappa}(\bar{\kappa}_n-\kappa_n) \{a(F v^\prime-Eu^\prime)+b(Gv^\prime-Fu^\prime)\}.
\end{eqnarray*}
This proves our claim.
\end{proof}

\begin{corollary}
Let $J:{\bf M}\rightarrow {\overline{\bf M}}$ be an isometry and $\alpha(s)$ be a normal curve on ${\bf M}$. Then the component of the normal curve $\alpha(s)$ along $\bf T\times\bf N$ is invariant if and only if any one of the following holds:
\begin{itemize}
\item[(i)]The position vector of $\alpha(s)$ is in the normal direction of $\alpha$.
\item[(ii)] The normal curvature is invariant.
\end{itemize}
\end{corollary}
\begin{proof}
From $(\ref{3.6.C})$,  $\bar{\alpha}\cdot(\bar{\bf T}\times\bar{\bf N})=\alpha\cdot(\bf T\times\bf N)$ if and only if 
\begin{eqnarray*}
\frac{\mu}{\kappa}(\bar{\kappa}_n-\kappa_n) \{a(Fv^\prime-Eu^\prime)+b(Gv^\prime-Fu^\prime)\}=0\\
\text{i.e., if and only if } \mu=0 \text{ or }\bar{\kappa}_n-\kappa_n=0 .
\end{eqnarray*}
If $\mu=0$ then from $(\ref{1})$, we see that $\alpha(s)=\lambda(s)n(s)$, i.e., the position vector of the normal curve $\alpha(s)$ is in the normal direction of itself.
\end{proof}
\begin{corollary}
Let $J:{\bf M}\rightarrow {\overline{\bf M}}$ be an isometry and $\alpha(s)$ be a normal curve on ${\bf M}$. The component of the normal curve $\alpha(s)$ along $\bf T\times\bf N$ is invariant and the position vector of $\alpha(s)$ is not in the normal direction of $\alpha$, then $\alpha(s)$ is asymptotic if and only if $\bar{\alpha}(s)$ is asymptotic.
\end{corollary}
\begin{proof}
From Corollary $3.9.$,  $\bar{\alpha}\cdot(\bar{\bf T}\times\bar{\bf N})=\alpha\cdot(\bf T\times\bf N)$ and the position vector of $\alpha(s)$ is not in the normal direction of $\alpha$ if and only if $\kappa_n=\bar{\kappa}_n$.\\
Therefore $\alpha(s)$ is asymptotic if and only if $\kappa_n=0$ if and only if $\bar{\kappa}_n=0$ if and only if $\bar{\alpha}(s)$ is asymptotic.
\end{proof}

\begin{proposition}
The geodesic curvature of a smooth curve and in particular of a normal curve remains invariant under isometry.
\end{proposition}
\begin{proof}
Let $\alpha$ be a curve on a parametric surface ${\bf M}$, then the geodesic curvature is given by Beltrami formula as:
\begin{equation}\label{3.18}
\kappa_g=\Big[\Gamma_{11}^2{u^\prime}^3+(2\Gamma_{12}^2-\Gamma_{11}^1){u^\prime}^2 v^\prime +(\Gamma_{22}^2-2\Gamma_{12}^1)u^\prime {v^\prime}^2-\Gamma_{22}^1{v^\prime}^3+u^\prime v^{\prime\prime}-u^{\prime \prime}v^\prime\Big] \sqrt{EG-F^2},
\end{equation}
where $\Gamma_{ij}^k$ are the Christoffel symbols of the second kind given by
\begin{equation}\label{3.19}\left\{
\begin{array}{ll}
\Gamma_{11}^1=\frac{1}{2W^2}\left\{GE_u+F[E_v-2F_u]\right\}, \quad \Gamma_{22}^2=\frac{1}{2W^2}\left\{EG_v+F[G_v-2F_v]\right\}\\
\Gamma_{11}^2=\frac{1}{2W^2}\left\{E[2F_u-E_v]-FE_v\right\},\quad \Gamma_{22}^1=\frac{1}{2W^2}\left\{G[2F_v-G_u]-FG_v\right\}\\
\Gamma_{12}^2=\frac{1}{2W^2}\left\{EG_u-FE_v\right\}=\Gamma_{21}^2,\quad
\Gamma_{21}^1=\frac{1}{2W^2}\left\{GE_v-FG_u\right\}=\Gamma_{12}^1
\end{array}
\right.
\end{equation}
and $W=\sqrt{EG-F^2}$. Thus, in view of (\ref{1.15}), (\ref{3.18}) and (\ref{3.19}), we see that $\overline{\kappa}_g =\kappa_g$. In particular the same holds for a normal curve.
\end{proof}
\section{acknowledgment}
 The third author greatly acknowledges to The University Grants Commission, Government of India for the award of Junior Research Fellow.

\end{document}